\documentclass[]{amsart}
\usepackage{a4wide}
\usepackage{amssymb}
\usepackage{bbm}

\usepackage{amssymb}

\usepackage{enumerate}
\usepackage[OT4]{fontenc}

\makeatletter
\@namedef{subjclassname@2010}{%

  \textup{2010} Mathematics Subject Classification}

\makeatother

\newtheorem{thm}{Theorem}[section]

\newtheorem{lem}[thm]{Lemma}

\newtheorem{prop}[thm]{Proposition}

\theoremstyle{definition}

\newtheorem*{xrem}{Remark}

\numberwithin{equation}{section}

\begin{document}

\title[On the parabolic equation]{On the parabolic equation for portfolio problems \thanks{Forthcoming in Banach Center Publications -- Conference on stochastic modeling in finance and insurance, B\k{e}dlewo 11.02.2019--15.02.2019, X Simons Semester}}

\thanks{Forthcoming in Banach Center Publications -- Conference on stochastic modeling in finance and insurance, B\k{e}dlewo 11.02.2019--15.02.2019, X Simons Semester}



\maketitle

\begin{center}
 Dariusz Zawisza  \footnote[2]{Faculty of Mathematics and Computer Science,   Jagiellonian University, {\L}ojasiewicza  6, 30-348 Krak{\'o}w, Poland, dariusz.zawisza@im.uj.edu.pl}
\end{center}

\begin{abstract}
We consider a semilinear equation linked to the finite horizon 
consumption -- investment problem under the stochastic factor framework and we prove it admits a classical solution and provide all obligatory estimates to successfully apply a verification reasoning. The paper covers the standard time additive utility, as well as the recursive utility framework. We extend existing results by considering more general factor dynamics including a non-trivial diffusion part and a stochastic correlation between assets and factors. In addition, this is the first paper which compromises many other optimization problems in finance, for example those related to the indifference pricing or the quadratic hedging problem. The extension of the result to the stochastic differential utility and robust portfolio optimization is provided as well. The essence of our paper lays in using improved stochastic methods to prove gradient estimates for suitable HJB equations with restricted control space.
\end{abstract}

\section{Introduction}  The use of a stochastic factor model in optimal portfolio selection problems has become recently very popular. A stochastic factor is often used to model stochastic patterns in the mean and the variance of financial returns. The topic has been explored under many different assumptions and many different investor's objectives. In the current paper we are interested in a semilinear parabolic partial differential equation which arise naturally in many consumption investment problems under the recursive utility formulation (with the Epstein-Zin utility and after some reduction techniques), whenever we assume a  stochastic factor dependence in the asset price dynamics. Our setting covers as well many other optimization topics. In particular, we should emphasize here the minimal variance martingale measure (generally q-optimal martingale measure), the minimal entropy martingale measure and the quadratic hedging problem. In order to determine an optimal investment strategy and use a verification theorem, the common approach is to prove that the PDE admits a classical solution in the $C^{2,1}$ class, which in addition satisfies a global gradient estimate.

 The regularity of a solution to the suitable PDE was explored by many authors, but due to our knowledge there is no paper considering it with the full generality including a multidimensional factor dynamics with a stochastic correlation, a non-trivial diffusion part and under the recursive utility formulation with a risk averse and a risk seeking investor. The difficulty with a stochastic correlation lays in the fact that it is not possible to remove the quadratic gradient part using the power transform  or obtain satisfactory results with the log transform techniques. Some results, concerning regularity of a related HJB equation, were obtained for example by Pham \cite{Pham}, Zariphopoulou \cite{Zariphopoulou} but they considered only the pure investment problem (without the consumption process) and with the trivial diffusion part in the factor dynamics or with one dimensional factor only. 
  
 A separate study should be dedicated to the line of papers concerning optimal investments with the risk sensitive objective criterion. Usually, those papers consider the infinite horizon formulation but they provide as well some insights into the finite horizon framework. This was considered by   
by  Bensoussan et al.\cite{Besnousan}, Hata \cite{Hata2},  Nagai \cite{Nagai}. We should emphasize here the significance of the papers written by Davis and Lleo \cite{Davis1}, \cite{Davis2} dedicated to very general finite horizon jump diffusion models. 

Smooth solutions to  pure investment problems can be as well easily deduced from Sobolev's weak solutions results obtained by BSDE methods (see e.g. Delarue and Guatteri \cite{Delarue}). 

Under the one dimensional factor dynamics, the possibility of a consumption was considered by Constaneda-Leyva and Hern\'{a}ndez - Hern\'{a}ndez \cite{Hernandez} (in the time additive utility setting), Kraft et al. \cite{Kraft}, \cite{Kraft2} (recursive utility),  and by Berdjane and Pergamentschikov \cite{Berdjane} with the multidimensional factor process, but still the factor dynamics is required to be suitable to reduce the quadratic gradient term by using the power transform in the spirit of Zariphopoulou \cite{Zariphopoulou}. As indicated in Pham \cite[Remark 3.1]{Pham} the aforementioned transform can not be used in a general factor dynamics.

 We should mention here as well many research papers embedded into the infinite horizon setting with the elliptic HJB equation instead of the parabolic one. We have here:  Hern\'{a}ndez and Fleming \cite{Fleming}, \cite{Fleming2},  Nagai \cite{Nagai}, Trybu{\l}a \cite{Trybula2}, Zawisza \cite{Zawisza}. Mostly they were focused on the one dimensional factor model, except for works of Hata and Sheu \cite{Hata} and Nagai \cite{Nagai}. They  obtained their results by applying  sub- and super-solution viscosity methods to deal with a general and a multidimensional factor dynamics. However, to prove existence result for the suitable viscosity solution they needed  higher regularity assumptions for the model coefficients. 
 
 Just recently Hata at al. \cite{Hata3} have considered the finite horizon case together with risk averse case, but still under higher smoothness conditions on the coefficients and within the framework of sub/supersolution method.  We are as well aware of recent very general results of Xing \cite{Xing}, Matoussi and Xing \cite{Matoussi}, which represent the solution in terms of the forward-backward equation. Nevertheless, they do not study associated parabolic equations.

Our results extend the recent results of Kraft et al. \cite{Kraft} and Hata, et al. \cite{Hata3}  but our approach to the issue is different. In this paper we rewrite the semilinear equation as the HJB equation with unrestricted control space, then we will restrict the control space to some compact set and use known existence theorems for such HJB equations. Using novel stochastic methods we obtain uniform estimates for the solution and its gradient. In this way we are able to prove that the solution to the restricted control problem is in fact a solution to our primary equation and as a by-product we get the estimates needed to apply the verification reasoning. 
Our paper is the continuation of the line of papers: Fleming and McEneaney \cite{mcefle},  Fleming and Hern\'{a}ndez \cite{Fleming}, Pham \cite{Pham}, Zawisza \cite{Zawisza}, but we present as well novel ideas. Namely, the most challenging problem in our work is to improve existing gradient stochastic estimates to cover many issues: the non-trivial second order term (diffusion term), the quadratic dependence in the gradient part,  the presence of the power expression in the equation and finally different configurations of  risk aversion parameters. The power expression corresponds to the presence of the consumption in the original control problem and prevents us from applying  the logarithmic  transform
 as it was done for example in Pham \cite{Pham}. The second important extension is a link to the variance hedging problems.
  We would like to point out as well all advanatages of our paper:  
\begin{enumerate}
\item We prove regularity results for nontrivial PDE which is useful to tackle the multidimensional portfolio problems including the Epstein -- Zin recursive utility  problem (vel. the Kreps-Proteus recursive utility problem)   and the quadratic hedging problem.  
\item We reduce many financial optimization problems to the analysis of one single semilinear PDE.
\item We provide a general method to prove global gradient estimates for solutions to the aforementioned equations (see Lemma \ref{lemma}).
\item In Lemma \ref{interestinglemma} we show an interesting application of the power transform method. It can have possible applications in other risk sensitive control problems.  
\item We reduce the equation to the form $u_{t}+ \frac{1}{2} Tr(\Sigma(x) D^{2}_{x} u)+H(D_{x}u,u,x,t)=0$, where the Hamiltonian $H(p,u,x,t)$ satisfies the Lipschitz condition in $u$ and $p$ and therefore our result paves the way to numerical simulation either by the fixed point method in partial differential equations (see e.g. Zawisza \cite{Zawisza}),  BSDE numerical techniques (see Gobet et al. \cite{Gobet}), or policy iteration algorithm (see Jacka and Mijatovi\'{c} \cite{Jacka}).
\item We present extensions of our result to the robust optimal portfolio selection problems 
\item The paper has implications for the existence of the stochastic differential utility under the  Kreps--Proteus utility.
\item Finally, we provide a proof which is independent of BSDE theory and might be use to prove existence theorems for forward--backward equations.
\end{enumerate}

Our paper has the following structure. First we introduce our equation and present the power transform method. Crucial estimates for the solution to the HJB equation are done in the third section. The main result (Theorem \ref{maintheorem}) is presented in the section number 4. In the last section we extend the main result to the two important problems: the robust portfolio optimization and the stochastic differential utility problem.
\section{The equation}

We consider first the Cauchy problem of the form 
\begin{align} \notag
&G_{t}+ \frac{1}{2} Tr(\Sigma(x) D^{2}_{x} G) + \frac{1}{2} \frac{1}{G} D_{x}^{*} G  A(x) D_{x} G + b^{*}(x) D_{x} G  \\ &\;+  \theta (1-k) G^{k} + h(x)G=0, \qquad (x,t) \in \mathbb{R}^{n} \times [0,T),  \label{maineq}
\end{align}
with the  terminal condition $G(x,T)=\beta(x)$ and parameters
 $k \in \mathbb{R}$, $\theta \geq 0$. 
In the above notation we ignore $(x,t)$ dependence for the function $G$. The lack of the time dependence in the model coefficients is for notational convenience and can be relaxed. In addition, vectors are treated as column matrices, the symbol $b^{*}$ is used to denote the transpose of the vector (or matrix) $b$. The above equation is a general version of the equation considered by Zawisza \cite{Zawisza},  Trybu{\l}a and Zawisza \cite{Trybula} and covers multidimensional setting proposed for example in Hata et al. \cite{Hata2}, Hata \cite{Hata3} (after taking the log transformation of our equation). 

We assume here that: 

{\it A1)} The  function $b:\mathbb{R}^{n} \to \mathbb{R}^{n} $ is Lipschitz continuous, while the function  $h:\mathbb{R}^{n} \to \mathbb{R}$ is Lipschitz continuous and bounded.

{\it A2)} The function $\beta: \mathbb{R}^{n} \to \mathbb{R}$ is positive, bounded, Lipschitz continuous and bounded away from zero i.e. there exists a constant $\varepsilon >0$ such that
\[
\beta (x) \geq \varepsilon > 0, \quad x \in \mathbb{R}^{n}.
\] 

{\it A3)} The matrices $\Sigma: \mathbb{R}^{n} \to \mathbb{R}^{n\times n}$ and $A: \mathbb{R}^{n} \to  \mathbb{R}^{n\times n}$ are symmetric, their coefficients are Lipschitz continuous and bounded, and the uniform ellipticity condition holds i.e. there exists a positive constant $\varepsilon>0$ such that  
\[
z^{*} \Sigma(x) z \geq \varepsilon  |z|^{2}, \quad z^{*} A(x) z \geq \varepsilon  |z|^{2} \quad x,z \in \mathbb{R}^{n}.
\] 

In many financial optimization problems the matrix $A$ is negative definite, so we will further prove that our results can be easily extended to the following case:

{\it A3')} The matrices $\Sigma: \mathbb{R}^{n} \to  \mathbb{R}^{n\times n}$ and $A: \mathbb{R}^{n} \to \mathbb{R}^{n\times n}$ are symmetric, their coefficients are Lipschitz continuous and bounded, and there exist $\varepsilon>0$, $\mu \in (0,1)$ such that  
\[
z^{*} \Sigma(x) z \geq \varepsilon  |z|^{2}, \quad z^{*} \left[\mu \Sigma(x)+A(x)\right] z \geq \varepsilon  |z|^{2} \quad x,z \in \mathbb{R}^{n}. 
\]

If we consider $\beta \equiv 1$, then the equation is suitable for the consumption - investment problem in a general stochastic factor model under the recursive utility formulation using continuous time Epstein -- Zin preferences. Our framework includes as well the standard time additive utility objectives.
The equation has already been derived in many papers (see for example Hata et al. \cite{Hata3}),  
so we will limit ourselves only to the one dimensional case as it is was done in Zawisza \cite{Zawisza}. First, we would like to present the constant correlation case
\begin{equation*} \label{model}
\begin{cases}
dB_{t} &=r(X_{t}) B_{t} dt,  \\
dS_{t} &=b(X_{t}) S_{t}  dt + \sigma(X_{t}) S_{t}  dW_{t}^{1},   \\
dX_{t} &=g(X_{t}) dt + a(X_{t})(\rho dW_{t}^{1} + \bar{\rho} dW_{t}^{2}),
\end{cases}
\end{equation*}
where $W^{1}$, $W^{2}$ are independent Wiener processes, $\rho$ is the constant correlation coefficient, the process $S$ denote the stock price, whereas $X$ is the factor process.
The wealth process is given by 
\[
dV_{t}^{\pi,c}= r(X_{t})V_{t}^{\pi,c} dt + \pi_{t} [b(X_{t}) -r(X_{t})]dt + \pi_{t} \sigma(X_{t}) dW_{t}^{1} - c_{t}dt. 
\]
By $(\pi,c)$ we denote the portfolio process and the consumption intensity process respectively i.e. a pair of progressively measurable processes such that the process $(c_{t},\; t \in [0,T])$ is positive,
\[
\int_{0}^{T} \pi_{s}^{2} ds <+\infty, \quad \text{a.s.}
\]
and the random variable $V_{T}^{\pi,c}$ is almost surely positive.
  The investor's objective is to maximize
\[\mathbb{E}_{v,x,t} \frac{1}{\gamma} \left[\int_{t}^{T}e^{-w (s-t)} (c_{s})^{\gamma} ds + e^{-w(T-t)}(V_{T}^{\pi,c})^{\gamma}\right]. \]
The HJB equation associated with that problem (after suitable reduction techniques - see e.g. Zariphopoulou \cite{Zariphopoulou}) is given by  
\begin{align*}
 & F_{t}  + \frac{1}{2} a^2(x)  D_{x}^{2} F +\frac{\gamma \rho^{2}}{2(1-\gamma)}  a^{2}(x) \frac{[D_{x}F]^2}{F} +  \left[g(x) +  \frac{\gamma \rho}{1-\gamma} a(x) \lambda(x) \right]D_{x}F + \\ \quad & + \left[\frac{\gamma}{2(1-\gamma)} \lambda^{2}(x)  + \gamma r(x) -w \right]F + (1-\gamma) F^{\frac{\gamma}{\gamma-1}} \notag
 =0,
\end{align*}
where $\lambda(x):=(b(x)-r(x))/\sigma(x)$ is a market price of risk.

In this case the power transform $F^{\zeta}$ (for suitable choice of $\zeta$) can be used to reduce the nonlinear term $\frac{[D_{x}F]^2}{F}$. The candidate optimal controls are given by
\[\hat{\pi}(V_{t},X_{t}) :=V_{t}\left[\frac{\rho  a(X_{t})}{(1-\gamma)  \sigma(X_{t}) } \frac{D_{x} F}{F}+ \frac{\lambda(X_{t})}{(1-\gamma)  \sigma (X_{t})} \right],  \quad  \hat{c}(V_{t},X_{t}):=V_{t} [F(X_{t},t)]^{\frac{1}{\gamma-1}}.
\]

However, instead of a deterministic correlation, we can consider a stochastic correlation effect i.e.
\[
dX_{t}=g(X_{t}) dt + a_{1}(X_{t}) dW_{t}^{1} + a_{2}(X_{t}) dW_{t}^{2}.
\] 
For that model the HJB equation is given by
\begin{align} \notag
 &F_{t}  + \frac{1}{2}[a_{1}^{2}(x) +a_{2}^{2}(x)] D^{2}_{x} F +\frac{\gamma}{2(1-\gamma)}  a^{2}_{1}(x) \frac{[D_{x} F]^2}{F} +  \left[g(x) +  \frac{\gamma }{1-\gamma} a_{1} (x) \lambda(x) \right]D_{x} F \\ \quad &+ \frac{\gamma}{2(1-\gamma)} \left[\lambda^{2}(x)   + \gamma r(x) -w \right]F + (1-\gamma) F^{\frac{-\gamma}{1-\gamma}}  \label{equation}
 =0
\end{align}
with the candidate optimal controls 
\[\hat{\pi}(V_{t},X_{t}) :=V_{t}\left[\frac{a_{1}(X_{t})}{(1-\gamma)  \sigma(X_{t}) } \frac{D_{x} F}{F}+ \frac{\lambda(X_{t})}{(1-\gamma)  \sigma (X_{t})} \right],  \quad   \hat{c}(V_{t},X_{t}):=V_{t} [F(X_{t},t)]^{\frac{1}{\gamma-1}}.
\]
In this example there is no further possibility to use the power transform to simplify the equation.

Although the fundamental  motivation for considering the problem comes from the above example, equation \eqref{maineq} is sufficiently general to cover and extend  many other important optimization problems in finance. We will try to review it once more by giving more information about the specific choice of the optimization objectives and the literature with the recent contribution in the field:
\begin{enumerate}
\item $\beta \equiv 1$ -- consumption investment problem for the recursive utility and the time additive utility agregator ( Kraft et al. \cite{Kraft2},  \cite{Kraft}, Hata et al. \cite{Hata3}).

\item $\beta \equiv 1$, $\theta =0$ -- the pure investment problem in the CRRA utility (HARA utility) framework (Davis and Lleo \cite{Davis1}, \cite{Davis2}).

\item  $\theta =0$ and condition {\it (A3')} -- indifference pricing under the exponential utility function, the minimal entropy martingale measure, (Benth and Karlsen \cite{Benth}, Hern\'{a}ndez--Hern\'{a}ndez and Sheu \cite{Hernandez2}, Sircar and Zariphopoulou \cite{Sircar}, Musiela and  Zaripho-\\-poulou \cite{Musiela}, Henderson \cite{Henderson}, Benedetti and Campi \cite{Benedetti}, Grasselli  Hurd \cite{Grasselli}, Zawisza \cite{Zawisza4}).

\item For $\beta \equiv 1$, $\theta =0$, condition {\it (A3')} -- the mean variance portfolio selection problem, the variance optimal martingale measure, (Hern\'{a}ndez - Hern\'{a}ndez \cite{Hernandez3}, Laurent and Pham \cite{Laurent}, Trybu{\l}a and Zawisza \cite{Trybula} ). These authors have considered so far only problems with one dimensional factor dynamics. 
\end{enumerate}

\begin{xrem}
It might happen that in the quadratic hedging and  the mean variance hedging problem both conditions {\it (A3)} and {\it (A3')} are not satisfied. But then the substitution $H=G^{-1}$ allows us to make condition {\it (A3')} applicable (see Trybu{\l}a and Zawisza \cite{Trybula}).
\end{xrem}

It is useful to note that we can restrict ourselves  to the equation with parameter $k \in (- \infty,0) \cup (1,+\infty)$ and instead of the condition {\it ( A3')} we can consider only ${\it (A3)}$. 


To prove this observation suppose first that condition {\it (A3')} is satisfied. Let us define 
\[
B_{\psi}(x):= \psi A(x) +  (\psi-1)\Sigma(x) = \frac{1}{(1-\mu)} [A(x) + \mu \Sigma(x) ], 
\]
where $\psi := \frac{1}{1-\mu}$ ($\mu$ as in {\it (A3')}). Note that
\[
 B_{\psi}(x) = \frac{1}{(1-\mu)} [A(x) + \mu \Sigma(x)],
\]
which implies that the matrix $B_{\psi}$ satisfies condition $\it A3$. Hence, we have
\begin{lem} \label{interestinglemma}
Suppose that condition {\it A3'} is satisfied and the function $G$ is a classical solution to the equation

\begin{align*}
&G_{t}+ \frac{1}{2} Tr(\Sigma(x) D^{2}_{x} G) + \frac{1}{2} \frac{1}{G} D_{x}^{*} G  B_{\psi}(x) D_{x} G + b^{*}(x) D_{x} G\\ & \;+  \frac{\theta}{\psi}(1-k) G^{k} + \frac{h(x)}{\psi} G=0, \quad \quad (x,t) \in \mathbb{R}^{n} \times [0,T).
\end{align*}
Then the function $H=G^{\psi}$ is a solution to 
\begin{align} \notag
&H_{t}+ \frac{1}{2} Tr(\Sigma(x) D^{2}_{x} H) + \frac{1}{2} \frac{1}{H} D_{x}^{*} H  A(x) D_{x} H + b^{*}(x) D_{x} H \\ & \; +\theta \frac{1-k}{1-\xi} (1-\xi) H^{\xi} + h(x)H=0, \quad  (x,t) \in \mathbb{R}^{n} \times [0,T), \label{aftersubstitution}
\end{align}
where
 $\xi := \frac{1}{\psi} \left[ k+ \psi -1\right]= 1+ \frac{k-1}{\psi} $. Moreover,
 \begin{enumerate}
\item $\xi>1$ if and only if $k>1$,
\item $\xi <0$ implies $k<0$,
\item $\xi \in [0,1)$ we can get by taking $\psi$ sufficiently large ($\mu$ close to one) and then take suitable $k<0$,
\item The case $\xi=1$ can be treated by considering the case $\theta=0$.
\end{enumerate}
\end{lem}
\begin{proof}
We have
\begin{align*}
H_{t}&= \psi G^{\psi-1} G_{t}, \\
D_{x}H &= \psi G^{\psi-1} D_{x}G, \\
H_{x_{i}x_{j}}&= \psi G^{\psi-1} G_{x_{i}x_{j}} + \psi (\psi-1) G^{\psi-2}G_{x_{i}}G_{x_{j}}.
\end{align*}
By the direct substitution, we get equation \eqref{aftersubstitution}. The latter part of the conclusion is left to the reader.
\end{proof}

For the one--dimensional  applications of the power transform method see Musiela and Zariphopoulou \cite{Musiela}, Zariphopoulou \cite{Zariphopoulou}, Zawisza \cite{Zawisza}, and Trybu{\l}a and Zawisza \cite{Trybula}, Kraft et al. \cite{Kraft}.

\section{Stochastic estimates}
In the current section we reduce this equation to the HJB equation  for the restricted stochastic control problem and prove estimates for the solution $F$ and its derivative $D_{x} F$. As it has been observed in the previous section we can limit ourselves to prove there exists a smooth solution to
\begin{align*} 
&G_{t}+ \frac{1}{2} Tr(\Sigma(x) D^{2}_{x} G) + \frac{1}{2} \frac{1}{G} D_{x}^{*} G  A(x) D_{x} G + b^{*}(x) D_{x} G \\ &\;+  \theta (1-k)  G^{k} + h(x)G=0, \quad (x,t) \in \mathbb{R}^{n} \times [0,T), 
\end{align*}
with the parameter restriction $ k \in (-\infty,0) \cup (1,+\infty)$. We will consider cases $k \in (-\infty,0)$ and $k \in (1,+\infty)$ separately.

{\bf Case 1} $k<0$.

First, we should noticed that it is well known that if the matrix is symmetric and positive definite, then there exists the unique positive square root of the matrix, which is also symmetric. Moreover, if the coefficients are bounded, uniformly  Lipschitz  continuous  and the uniform ellipticity condition holds then  the same is true for the square root (cf. Stroock and Varadhan \cite[Lemma 5.2.1 and Theorem 5.2.2]{Stroock}). Thus, let $\sigma$ denote the unique square root of $\Sigma$ and let $V$ denote  the square root of $A$.

Suppose first that there exists a positive solution $G$ to equation \eqref{maineq} such that
\[
m_{1}\leq G^{\frac{1}{\alpha-1}} \leq m_{2},  \quad \left|\frac{D_{x}^{*}G V(x)}{G} \right| \leq R,
\]
for some $0 \leq m_{1}<1<m_{2},\;R>0$, where $0<\alpha<1$ is a constant determined by the formula $\frac{\alpha}{\alpha-1}=k$.
Then 
\begin{equation} \label{forem}
(1-\alpha) G^{\frac{\alpha}{\alpha-1}} = \max_{c \geq 0} \left(-\alpha  c G + c^{\alpha} \right) = \max_{c \in [m_{1},m_2]} \left(-\alpha  c G + c^{\alpha} \right), 
\end{equation}
and 

\begin{align} \notag
&\frac{1}{2} \frac{1}{G} D_{x}^{*} G A(x) D_{x} G  = \max_{q \in \mathbb{R}^{n}} \left(D_{x}^{*}G  V(x) q    - \frac{1}{2}  |q|^{2} G \right)\\ & \quad=\max_{q \in B_{R}} \left(  D_{x}^{*} G V(x) q  - \frac{1}{2}  |q|^{2} G \right), \label{forde}
\end{align}
where $B_{R}$ is the compact set $\{q \in \mathbb{R}^{n}: |q| \leq R \}$.
Therefore, this is the motivation to consider first the HJB equation for a control problem of the form
\begin{align} \label{hjb}
&G_{t}+ \frac{1}{2} Tr (\Sigma(x) D_{x}^{2} G) + b^{*}(x) G_{x}+ \max_{q \in B_{R} } \left(  D_{x}^{*} G V(x) q- \frac{1}{2} |q|^{2} G\right) \\ &\;+ \theta \max_{m_{1}\leq c \leq m_{2}} \left(-\alpha  c G + c^{\alpha} \right) + h(x) G=0, \notag
\end{align}
with the terminal condition $G(x,T)=\beta(x)$. In fact, to keep consistency of the notation we should use in \eqref{hjb} the term 
\[
\frac{\theta}{(1-\alpha)^{2}}  \max_{m_{1}\leq c \leq m_{2}} \left(-\alpha  c G + c^{\alpha} \right).
\]
Nonetheless, the term $\frac{\theta}{(1-\alpha)^{2}}$ is positive, so  for notational convenience and without loss of generality and we can simply replace it by $\theta$.

Assuming conditions {\it (A1)}--{\it (A3)} and using  Zawisza \cite[Theorem 2.3]{Zawisza2}, we know that equation \eqref{hjb} has a smooth solution and we will denote it by $G_{m_{1},m_{2},R}$ (alternatively we may use  $W^{2,1}$ very general results proved by Delarue and Guatteri \cite{Delarue} but it includes  only the  bounded coefficients case and consequently it does not cover full generality of our paper). By the standard verification theorem, we have
\begin{multline*}
G_{m_{1},m_{2},R}(x,t)  =  \sup_{ q \in \mathcal{A}_{R}, c \in \mathcal{C}_{m_{1},m_{2}}} \mathbb{E}_{x,t} \biggl[ \int_{t}^{T}\theta e^{\int_{t}^{s}(h(X_{k}^{q}(x,t))- \frac{1}{2} |q_{k}|^{2}- \theta \alpha c_{k})\, dk} c_{s}^{\alpha} ds  \\ + e^{\int_{t}^{T}(h(X_{k}^{q}(x,t))- \frac{1}{2} |q_{k}|^{2}- \theta \alpha c_{k})\, dk} \beta(X_{T}^{q}(x,t))\biggr],
\end{multline*}
where
\begin{equation} \label{sde}
dX_{k}^{q} = [b(X_{k}^{q}) - V(X_{k}^{q}) q_{k}]  dk + \sigma(X_{k}^{q})dW_{k}, 
\end{equation}
and $(W_{k}=(W_{k}^{1},W_{k}^{2},\ldots,W_{k}^{n})^{*}), 0 \leq k \leq T)$ is a $n$-dimensional Brownian motion. Note, that to apply properly the verification reasoning we need the existence of the solution to SDE \eqref{sde}  when $q$ is a feedback control. Here we can use the result proved by Gy\"{o}ngy and Krylov \cite[Corollary 2.6]{Krylov}. In the above stochastic control representation the symbol $\mathcal{A}_{R}$ is used to denote all progressively measurable processes $(q_{s}, 0 \leq s \leq T)$ taking values in $B_{R}$, the symbol $\mathcal{C}_{m_{1},m_{2}}$ to denote all progressively measurable processes taking values in $[m_{1},m_{2}]$. In addition, by using $\mathbb{E}_{x,t}f(X_{s})$ we stress the fact we take the expected value of suitable random variable, when the system starts from $x$ at time $t$. In future, for notational convenience, we will often write  $\mathbb{E}f(X_{s}(x,t))$. 

Our aim is now to prove that we can find the constant $\widehat{R}>0$ and $\widehat{m}_{1}>0$ and $\widehat{m}_{2}>0$ such that for $\widehat{G}:=G_{\widehat{m}_{1},\widehat{m}_{2},\widehat{R}}$, 
we have

\begin{equation} \label{monemtwo}
\left|\frac{D_{x}^{*} \widehat{G} V(x)}{\widehat{G}}  \right| \leq \widehat{R}, \quad \widehat{m}_{1} \leq [\widehat{G}]^{\frac{1}{1-\alpha}} \leq  \widehat{m}_{2}, \quad (x,t) \in \mathbb{R}^{n} \times [0,T).
\end{equation}

In that case, using equation \eqref{forem} and \eqref{forde}, we will be sure that $\widehat{G}$ is as well the solution to \eqref{maineq}. To find such parameters we need first to find the upper and the lower uniform bound for $G_{m_{1},m_{2},R}$ and the uniform bound for $|D_{x} G_{m_{1},m_{2},R}|$.
We start by proving uniform bounds for $G_{m_{1},m_{2},R}(x,t)$.

\begin{prop} \label{prop11}
Suppose that $k<0$, conditions listed in {\it (A1)} -- {\it (A3)} are satisfied and $G_{m_{1},m_{2},R}(x,t)$ is a classical solution to equation \eqref{hjb}. Then there exist $D_{1},D_{2}>0$ such that 
\[
D_{2} \leq G_{m_{1},m_{2},R}(x,t) \leq D_{1}, \quad (x,t) \in \mathbb{R}^{n} \times [0,T], \; m_{1} \leq 1 \leq m_{2}.
\]
\end{prop}

\begin{proof}
Since functions $h$ and $\beta$ are bounded there exists a constant $D>0$ such that  

\begin{multline*}
|G_{m_{1},m_{2},R}(x,t)|  \leq   \sup_{q \in \mathcal{A}_{R}, c \in \mathcal{C}_{m_{1},m_{2}}} \mathbb{E}_{x,t} \biggl[ \int_{t}^{T}\theta e^{\int_{t}^{s}(h(X_{k}^{q}(x,t))- \frac{1}{2} |q_{k}|^{2}- \theta \alpha c_{k})\, dk} c_{s}^{\alpha} ds \\ + e^{\int_{t}^{T}(h(X_{k}^{q}(x,t))- \frac{1}{2} |q_{k}|^{2}- \theta \alpha c_{k})\, dk} \beta(X_{T}^{q}(x,t))\biggr] \\ \leq D  \sup_{c \in \mathcal{C}_{m_{1},m_{2}}}\mathbb{E}_{x,t} \left[ \int_{t}^{T}e^{-\int_{t}^{s}\theta \alpha c_{k}\, dk} c_{s}^{\alpha} ds + 1\right].
\end{multline*}
Furthermore, for $\alpha \in (0,1)$ we have  
\[
\int_{t}^{T}e^{-\int_{t}^{s} \theta \alpha c_{k}\, dk} c_{s}^{\alpha} ds \leq \int_{t}^{T}e^{-\int_{t}^{s}\theta \alpha c_{k}\, dk} \chi_{\{c_{s} \leq 1\}} ds + \int_{t}^{T}e^{-\int_{t}^{s}\theta \alpha c_{k}\, dk} c_{s}\chi_{\{c_{s} > 1\}}ds,
\]
and by the first fundamental theorem of calculus, we get
\[
\int_{t}^{T}e^{-\int_{t}^{s}\theta \alpha c_{k}\, dk} c_{s}ds = \frac{1}{\theta \alpha}\left[- e^{-\int_{t}^{s}\theta \alpha c_{k}\, dk} \right]_{t}^{T} =\frac{1}{\theta \alpha}\left[1- e^{-\int_{t}^{T}\theta \alpha c_{k}\, dk} \right].
\]
Consequently, 
\[
\int_{t}^{T}e^{-\int_{t}^{s} \theta \alpha c_{k}\, dk} c_{s}^{\alpha} ds \leq T+ \frac{1}{\theta \alpha}.
\]

Thus, there exists a constant $D_{1}>0$ such that
\begin{equation} \label{d1}
|G_{m_{1},m_{2},R}(x,t)| \leq D_{1}, \quad (x,t) \in \mathbb{R}^{n} \times [0,T]. 
\end{equation}

By substituting $ c \equiv 1$ and $q \equiv 0$ and using the fact that the function $h$ is bounded, $\beta$ is bounded away from zero, we get the lower bound for $|G_{m_{1},m_{2},R}(x,t)|$ i.e. there exists $D_{2}>0$ such that for all $(x,t) \in \mathbb{R}^{n} \times [0,T]$
\begin{equation} 
|G_{m_{1},m_{2},R}(x,t)| \geq  \mathbb{E}_{x,t} \left[ e^{\int_{t}^{T}(h(X_{k}^{q}(x,t))-\theta \alpha) \, dk} \beta(X_{T}^{q}(x,t))\right] \geq D_{2}>0. \label{d2}
\end{equation}
\end{proof}

{\bf Case II} $k>1$.

In this case we are interested in the solution to the equation 

\begin{align} \label{hjb2}
&G_{t}+ \frac{1}{2} Tr (\Sigma(x) D_{x}^{2} G) + b^{*}(x) D_{x} G+ \max_{q \in B_{R} } \left(  D_{x}^{*} G V(x) q- \frac{1}{2} |q|^{2} G\right) \\ &\;+ \theta  \min_{0\leq c \leq m_{2}} \left(-\alpha  c G + c^{\alpha} \right) + h(x) G=0. \notag
\end{align}
Note that in this case we have $\alpha >1$. 

\begin{prop} \label{prop22}
Assume  $k>1$, all conditions listed in {\it (A1)} -- {\it (A3)} hold true  and $G_{0,m_{2},R}(x,t)$ is a bounded classical solution to equation \eqref{hjb2}. Then there exist \\ $D_{1},D(m_{2})>0$ such that 
\[
D(m_{2}) \leq G_{0,m_{2},R}(x,t) \leq D_{1}, \quad (x,t) \in \mathbb{R}^{n} \times [0,T], \; m_{2}>0.
\]
\end{prop}
\begin{proof}
First we will find an upper bound for the function $G$.
Note that the expression \\ $\min_{0\leq c \leq m_{2}} \left(-\alpha  c G + c^{\alpha} \right)$ is negative and 
\[
\min_{0\leq c \leq m_{2}} \left(-\alpha  c G + c^{\alpha} \right) = \min_{0\leq c \leq m_{2}} \left(-\alpha  c  + c^{\alpha} G^{-1} \right) G.
\]
Therefore, the stochastic representation (of the game type) gives the form
\begin{multline*}
G_{0,m_{2},R}(x,t)  =  \sup_{ q \in \mathcal{A}_{R}} \inf_{c \in \mathcal{C}_{0,m_{2}}}\mathbb{E}_{x,t} \biggl[ \int_{t}^{T}\theta e^{\int_{t}^{s}(h(X_{k}^{q}(x,t))- \frac{1}{2} |q_{k}|^{2}- \theta \alpha c_{k}  + \theta c_{k}^{\alpha}  G^{-1}_{k} )\, dk} ds  \\ + e^{\int_{t}^{T}(h(X_{k}^{q}(x,t))- \frac{1}{2} |q_{k}|^{2}- \theta \alpha c_{k} + \theta c_{k}^{\alpha}  G^{-1}_{k})\, dk} \beta(X_{T}^{q}(x,t))\biggr].
\end{multline*}

In the above expression the infimum is taken over $\mathcal{C}_{0,m_{2}}$, so it is smaller then the expectation taken under the assumption that $m_{2}>0$ with the control $c \equiv 0$.
In that case, in the  above exponent, all expressions are bounded from above, so there must exist a constant $D_{1}>0$ such that 
\[
G_{0,m_{2},R}(x,t) \leq D_{1}, \quad (x,t) \in \mathbb{R}^{n} \times [0,T], \;  m_{2} > 0.
\]
In addition, we have $\alpha >1$, so
\[
0<G_{0,m_{2},R}(x,t)^{\frac{1}{\alpha-1}} \leq D_{1}^{\frac{1}{\alpha-1}}, \quad (x,t) \in \mathbb{R}^{n} \times [0,T], \; m_{2}>0
\]
and we can set $m_{2}=D_{1}^{\frac{1}{\alpha-1}}$. With those parameters fixed, with the help of the stochastic representation
\begin{multline*}
G_{0,m_{2},R}(x,t)  =  \sup_{ q \in \mathcal{A}_{R}} \inf_{c \in \mathcal{C}_{0,m_{2}}}\mathbb{E}_{x,t} \biggl[ \int_{t}^{T}\theta c_{s}^{\alpha} e^{\int_{t}^{s}(h(X_{k}^{q}(x,t))- \frac{1}{2} |q_{k}|^{2}- \theta \alpha c_{k}   )\, dk} ds \\ + e^{\int_{t}^{T}(h(X_{k}^{q}(x,t))- \frac{1}{2} |q_{k}|^{2}- \theta \alpha c_{k})\, dk} \beta(X_{T}^{q}(x,t))\biggr]
\end{multline*}
and substituting $q \equiv 0$ , we have
\[
G_{0,m_{2},R}(x,t) \geq \mathbb{E}_{x,t}  e^{\int_{t}^{T}(h(X_{k}^{q}(x,t))- \theta   \alpha m_{2})\, dk} \beta(X_{T}^{q}(x,t)).
\]
The function $\beta$ is bounded away from $0$ and $h$ is bounded. This ensures existence for the constant $D(m_{2})>0$ such that 
\[
G_{0,m_{2},R}(x,t) > D(m_{2}) >0, \quad (x,t) \in \mathbb{R}^{n} \times [0,T], \; m_{2}>0.
\]
\end{proof}

In both cases $k<0$ and $k>1$ we estimates for $|D_{x} G_{m_{1},m_{2},R}|$.  Usually, authors use the approach based  on using the It\^{o} formula and estimate only $|X_{k}^{q}(x,t) - X_{k}^{q}(\bar{x},t)|^2$ (see for example Fleming and McEneaney \cite[Lemma 4.1]{mcefle} or Zawisza \cite[Theorem 4.3]{Zawisza}) and this is insufficient in our setting. Therefore, we need the following Lemma.
\begin{lem} \label{lemma}
If functions  $V$, $\sigma$ are Lipschitz continuous and bounded, the function $b$ is Lipschitz continuous and processes $(X_{k}^{q}(x,t), t \leq k \leq T)$, $(X_{k}^{q}(\bar{x},t), t \leq k \leq T)$ are strong solutions to 
\[
dX_{k}^{q} = [b(X_{k}^{q}) - V(X_{k}^{q})q_{k}]  dk + \sigma(X_{k}^{q})dW_{k}, \quad 
\]
and $(q_{k}, k \leq T)$ takes his values in $B_{R}$, then there exists a constant $\tilde{L}>0$ such that for all $(q_{k}, k \leq T)$
\[
\mathbb{E}\sup_{t \leq k \leq T }e^{-\frac{1}{2} \int_{t}^{k} |q_{s}|^{2} ds}|X_{k}^{q}(x,t) - X_{k}^{q}(\bar{x},t)| \leq \tilde{L}|x-\bar{x}|, \quad x,\bar{x} \in \mathbb{R}^{n},\; t \in [0,T].
\]
Moreover, the constant $\tilde{L}$ does not depend on the choice of the radius $R>0$.
\end{lem}

\begin{proof}
Applying the It\^{o} formula, we get  
\begin{multline} \label{afterito}
e^{-\frac{1}{2} \int_{t}^{s} |q_{l}|^{2} dl}\left( X_{s}^{q}(x,t) - X_{s}^{q}(\bar{x},t) \right) = (x-\bar{x}) \\+ \int_{t}^{s}e^{-\frac{1}{2} \int_{t}^{k}|q_{l}|^{2} dl}\left[\zeta(X_{k}^{q}(x,t),q_{k})- \zeta(X_{k}^{q}(\bar{x},t),q_{k})  \right]dk \\+ 
\int_{t}^{s}e^{-\frac{1}{2} \int_{t}^{k}|q_{l}|^{2} dl} \left[\sigma(X_{k}^{q}(x,t)) - \sigma(X_{k}^{q}(\bar{x},t))\right]dW_{k},
\end{multline}
where 
\[
\zeta(x,q)= \left[b(x) -  V(x)q - \frac{1}{2} |q|^{2} x  \right].
\]
Taking the maximum in $q$ yields the existence  of a constant $M >0$ such that for all  $x,\bar{x} \in \mathbb{R}^{n}$
\begin{multline}
(x-\bar{x})^{*}(\zeta(x,q) -\zeta(\bar{x},q)) \\= (x-\bar{x})^{*}(b(x)-b(\bar{x})) -  (x-\bar{x})^{*}(V(x)-V(\bar{x}))q - \frac{1}{2} |q|^{2} |x-\bar{x}|^{2}  \leq M |x-\bar{x}|^{2}.  
\end{multline}
Applying the It\^{o} formula once again, using the process \eqref{afterito} and the quadratic function, we have
\begin{multline*}
e^{-2 \cdot \frac{1}{2} \int_{t}^{s} |q_{k}|^{2} dk}\left| X_{s}^{q}(x,t) - X_{s}^{q}(\bar{x},t) \right|^{2} = |x-\bar{x}|^{2} \\+ \int_{t}^{s}2e^{-2 \cdot\frac{1}{2} \int_{t}^{k}|q_{l}|^{2} dl}[X_{k}^{q}(x,t)- X_{k}^{q}(\bar{x},t)]^{*}\left[\zeta(X_{k}^{q}(x,t),q_{k})- \zeta(X_{k}^{q}(\bar{x},t),q_{k})  \right]dk \\+\int_{t}^{s}e^{-2 \cdot\frac{1}{2} \int_{t}^{k}|q_{l}|^{2} dl}Tr \left( \left[\sigma(X_{k}^{q}(x,t)) -\sigma(X_{k}^{q}(\bar{x},t)) \right]^{*} \left[\sigma(X_{k}^{q}(x,t)) -\sigma(X_{k}^{q}(\bar{x},t)) \right] \right) dk  \\ +
\int_{t}^{s}2e^{-2 \cdot\frac{1}{2} \int_{t}^{k}|q_{l}|^{2} dl}[X_{k}^{q}(x,t)- X_{k}^{q}(\bar{x},t)]^{*} \left[\sigma(X_{k}^{q}(x,t)) - \sigma(X_{k}^{q}(\bar{x},t))\right]dW_{k}.
\end{multline*}
If the process $(q_{k}, t \leq k \leq T)$ takes his values in $B_{R}$, then by the standard estimates for the controlled processes (cf. Pham \cite[Theorem 1.3.15]{Pham2}), we have
\[
\mathbb{E} \sup_{t \leq k \leq T} |X_{k}^{q}(x,t)- X_{k}^{q}(\bar{x},t)|^{4}< +\infty
\]
and consequently, the process
\[
Z_{s}=\int_{t}^{s}e^{-2 \cdot\frac{1}{2} \int_{t}^{k}|q_{l}|^{2} dl}(X_{k}^{q}(x,t)- X_{k}^{q}(\bar{x},t))^{*} \left[\sigma(X_{k}^{q}(x,t)) - \sigma(X_{k}^{q}(\bar{x},t))\right]dW_{k}
\]
is a square integrable martingale. Using the martingale inequality, we get
\begin{multline}
\mathbb{E}\left[\int_{t}^{s}e^{-2 \cdot\frac{1}{2} \int_{t}^{k}|q_{l}|^{2} dl}[X_{k}^{q}(x,t)- X_{k}^{q}(\bar{x},t)]^{*} \left[\sigma(X_{k}^{q}(x,t)) - \sigma(X_{k}^{q}(\bar{x},t))\right]dW_{k}\right]^{2} \\ \leq M \int_{t}^{s} \mathbb{E} \sup_{0 \leq k \leq p} e^{-2 \cdot\frac{1}{2} \int_{t}^{k}|q_{l}|^{2} dl}|X_{k}^{q}(x,t)- X_{k}^{q}(\bar{x},t)|^{2} dp,
\end{multline}
for some constant $M>0$.
For fixed $x, \bar{x} \in \mathbb{R}$ define new function
\[
v(p) := \mathbb{E} \sup_{0 \leq k \leq p} e^{-2 \cdot\frac{1}{2} \int_{t}^{k}|q_{l}|^{2} dl}|X_{k}^{q}(x,t)- X_{k}^{q}(\bar{x},t)|^{2}.
\]

The above inequalities and the Lipschitz continuity of $\sigma$ ensure that  there exists a constant $\tilde{M}>0$ such that  
\[
v(s) \leq |x-\bar{x}|^{2} + \tilde{M}\int_{t}^{s} v(k) dk,   \quad x, \bar{x} \in \mathbb{R}^{n}.
\]
The Gronwall inequality yields
\[
v(s) \leq |x-\bar{x}|^{2} e^{\tilde{M}(s-t)}, \quad x, \bar{x} \in \mathbb{R}^{n}.
\]
\end{proof}

\section{Main theorem}
\begin{thm} \label{maintheorem}

Suppose that all conditions listed in {\it (A1)}, {\it (A2)}, {\it (A3)} or {\it (A1)}, {\it (A2)}, {\it (A3')} are satisfied. Then there exists $G \in \mathcal{C}^{2,1}(\mathbb{R} \times [0,T)) \cap \mathcal{C}(\mathbb{R} \times [0,T]) $, a classical positive solution  to equation \eqref{maineq}. In addition, the solution $G$ is uniformly bounded together with $D_{x} G$.
\end{thm}
\begin{proof}
It has already been proved (Proposition \ref{prop11} and Proposition \ref{prop22} ) that the analysis can be reduced only to the case $k< 0$ or $k>1$ and to the set of conditions {\it (A1)}--{\it (A3)}. Fortunately, we can consider both cases jointly because in both cases
we have already proved that there exists a suitable pair of constants $\widehat{m}_{1}<1<\widehat{m}_{2}$ such that 
\begin{equation} \label{e1}
\widehat{m}_{1} \leq [G_{\widehat{m}_{1},\widehat{m}_{2},R}]^{\frac{1}{1-\alpha}} \leq  \widehat{m}_{2}, \quad (x,t) \in \mathbb{R}^{n} \times [0,T],\; R>0.
\end{equation}

Now, we should find a uniform bound for $|D_{x} G_{m_{1},m_{2},R}|$. The method is based on finding the bound for the Lipschitz constant for the function $G_{m_{1},m_{2},R}$. To achieve our goal we use stochastic representation for our HJB equation. In both cases $k < 0$ and $k > 1$ we have
\begin{multline} \label{inequality}
\left|G_{m_{1},m_{2},R}(x,t)-G_{m_{1},m_{2},R}(\bar{x},t)\right| \\ \leq \sup_{ q \in \mathcal{A}_{R}, c \in \mathcal{C}_{m_{1},m_{2}}} \mathbb{E} \biggl[  \int_{t}^{T}\theta e^{-\int_{t}^{s}( \theta \alpha c_{k})\, dk} \\ \cdot \left|e^{\int_{t}^{s}(h(X_{k}^{q}(x,t))- \frac{1}{2} |q_{k}|^{2})\, dk}-e^{\int_{t}^{s}(h(X_{k}^{q}(\bar{x},t)- \frac{1}{2} |q_{k}|^{2})\, dk} \right|c_{s}^{\alpha} ds \\ + |\beta(X_{T}^{q}(x,t))|\left|e^{\int_{t}^{T}(h(X_{k}^{q}(x,t))- \frac{1}{2} |q_{k}|^{2}- \theta \alpha c_{k})\, dk}- e^{\int_{t}^{T}(h(X_{k}^{q}(\bar{x},t))- \frac{1}{2} |q_{k}|^{2}- \theta \alpha c_{k})\, dk}\right| \\+
e^{\int_{t}^{T}(h(X_{k}^{q}(\bar{x},t))- \frac{1}{2} |q_{k}|^{2}- \theta \alpha c_{k})\, dk}\left| \beta(X_{T}^{q}(x,t))- \beta(X_{T}^{q}(\bar{x},t)) \right| \biggr].
\end{multline}

Note, that the expression $\int_{t}^{s}(h(X_{k}^{q}(x,t))- \frac{1}{2} |q_{k}|^{2}- \theta \alpha c_{k})\, dk$ is bounded.  For the function $e^{x}$, we have $|e^{x}-e^{y}| \leq e^{z}|x-y|$,  if $x,y \leq z$. Thus, the first and the second expression on the right hand side of inequality \eqref{inequality} we can treat by taking the advantage of the following estimate 
\begin{align*}
&\left|  e^{\int_{t}^{s}(h(X_{k}^{q}(x,t))- \frac{1}{2} |q_{k}|^{2}- \theta \alpha c_{k})\, dk}- e^{\int_{t}^{s}(h(X_{k}^{q}(\bar{x},t))- \frac{1}{2} |q_{k}|^{2}- \theta \alpha c_{k})\, dk} \right| \\  & \quad \leq N_{1} e^{-\int_{t}^{s} \frac{1}{2} |q_{k}|^{2}\, dk}\left|\int_{t}^{s}h(X_{k}^{q}(x,t)) dk - \int_{t}^{s}h(X_{k}^{q}(\bar{x},t)) dk \right| \\ & \quad \leq N_{2} \sup_{t \leq k \leq T}e^{-\int_{t}^{k}\frac{1}{2} |q_{k}|^{2} dk} | X_{k}^{q}(x,t) - X_{k}^{q}(\bar{x},t)|,
\end{align*}
for suitable constants $N_{1}, N_{2}>0$.

Using the fact that $\beta$ is Lipschitz continuous and bounded and summarizing all inequalities together with Lemma \ref{lemma} we obtain the existence of  uniform constants $M_{1},M_{2}>0$ such that
\begin{align*}
&|G_{m_{1},m_{2},R}(x,t) - G_{m_{1},m_{2},R}(\bar{x},t)| \\&\quad  \leq M_{1}E\sup_{t \leq k \leq T}e^{-\int_{t}^{k}\frac{1}{2} |q_{l}|^{2} dl}\left | X_{k}(x,t) -  X_{k}(\bar{x},t)\right| \leq M_{2} |x-\bar{x}|
\end{align*}
and since coefficients of $V$ are bounded, $G_{m_{1},m_{2},R}$ is uniformly bounded from below, there exists 
 $\widehat{R}>0$ such that 
\begin{equation} \label{e2}
\left| \frac{D_{x}^{*} G_{m_{1},m_{2},R} V(x)}{G_{m_{1},m_{2},R}} \right| \leq  \widehat{R}, \quad  (x,t) \in \mathbb{R}^{n} \times [0,T], \; 0 \leq m_{1}  <m_{2}, \; R>0.
\end{equation}
Estimates \eqref{e1} and \eqref{e2} ensure that $G_{\widehat{m}_{1},\widehat{m}_{2},\widehat{R}}$ is the desired solution.
\end{proof}

\section{Extensions}

For notational convenience we have omitted  in the above analysis few possible extensions.
If we assume that parameter $\theta $ is state dependent we will get the equation important from the point of view of the so-called stochastic differential utility process with the Kreps - Proteus utility ( see Duffie and Epstein \cite{DuffieEpstein}, Duffie and Lions \cite{Duffie}). In this case, we have

\begin{align} \notag
&G_{t}+ \frac{1}{2} Tr(\Sigma(x) D^{2}_{x} G) + \frac{1}{2} \frac{1}{G} D_{x}^{*} G  A(x) D_{x} G +   b^{*}(x) D_{x} G \\ & \; +  \theta(x) (1-k) G^{k} + h(x)G=0, \quad \quad (x,t) \in \mathbb{R}^{n} \times [0,T). \label{secondequation}
\end{align}

We need the following condition: 

{\it A4)} The function $\theta : \mathbb{R} \to \mathbb{R}^{+}$  is Lipschitz continuous and bounded.  

\begin{thm} 
Suppose that all conditions listed in {\it (A1)}, {\it (A2)}, {\it (A3)}, {\it (A4)} or {\it (A1)}, {\it (A2)}, {\it (A3')}, {\it (A4)} are satisfied. Then there exists $G \in \mathcal{C}^{2,1}(\mathbb{R} \times [0,T)) \cap \mathcal{C}(\mathbb{R} \times [0,T]) $, a classical positive solution  to equation \eqref{secondequation}. In addition, the solution $G$ is uniformly bounded together with $D_{x} G$.
\end{thm}
 
\begin{proof} To complete the reasoning from the proof of Theorem \ref{maintheorem} it is worth to present only one estimate. Namely, we can again use the fact $|e^{x} - e^{y}| \leq e^{a}|x-y|$ for any $x,y \leq a$ and arrive at
\begin{align*}
& \left| e^{-\int_{t}^{s}\left( \alpha \theta(X_{k}^{q}(x,t)) c_{k} +  \frac{1}{2} |q_{k}|^{2}\right)\, dk}  - e^{-\int_{t}^{s}\left( \alpha \theta(X_{k}^{q}(\bar{x},t))  c_{k} + \frac{1}{2} |q_{k}|^{2}\right)\, dk} \right| \\
&  \leq L T |\alpha| \sup_{t \leq s \leq T}\left[ e^{-\frac{1}{2} \int_{t}^{s}   |q_{k}|^{2}\, dk}|X_{s}^{q}(x,t)-X_{s}^{q}(\bar{x},t)|  e^{-\int_{t}^{s}  \alpha \underline{\theta}  c_{k}\, dk}\int_{t}^{s} c_{k}dk\right],
\end{align*}
where $t \leq s \leq T$,\; $\underline{\theta}=\inf_{x \in \mathbb{R}} \theta(x)$ and $L>0$ is a Lipschitz constant for the function $\theta$. Now, it is sufficient to note that for any $\gamma >0$ the process
\[
e^{-\int_{t}^{s}  \gamma  c_{k}\, dk}\int_{t}^{s} c_{k}dk, \quad  \quad t \leq s \leq T
\]
is bounded because the function $xe^{-\gamma x}$ is bounded for $x>0$.
\end{proof} 

The second extension is dedicated to the robust optimal selection problem (see Zawisza \cite{Zawisza3}, \cite{Zawisza4}). We consider the equation
 
\begin{multline} \label{thirdequation}
G_{t}+ \frac{1}{2} Tr(\Sigma(x) D^{2}_{x} G) + \frac{1}{2} \frac{1}{G} D_{x}^{*} G  A(x) D_{x} G  \\ + \iota \min_{\eta \in \Gamma} \left( b^{*}(x,\eta) D_{x} G+  h(x,\eta)G\right)+ \theta(x) (1-k) G^{k} =0,  \quad (x,t) \in \mathbb{R}^{n} \times [0,T), 
\end{multline}
where $\Gamma \subset \mathbb{R}^{l}$ is a fixed compact set and $\iota \in \{-1,1\}$ is a parameter that allows to replace $\min$ with $\max$.
 
Here we need the following:

{\it A1')}  The  functions 
\[
b:\mathbb{R}^{n} \times \Gamma \to \mathbb{R}^{n}, \quad h:\mathbb{R}^{n} \times \Gamma \to \mathbb{R}, \quad \theta:\mathbb{R}^{n} \to \mathbb{R}^{+} \] are continuous (jointly in both variables) and  Lipschitz continuous in the first variable uniformly with respect to the second i.e. there exists a constant $L>0$ such that for all $x, \bar{x} \in \mathbb{R}$, $\eta \in \Gamma$ 
\begin{align*}
&|b(x,\eta) - b(\bar{x},\eta)| \leq L|x-\bar{x}|,\\
&|h(x,\eta) - h(\bar{x},\eta)| \leq L|x-\bar{x}|,\\
&|\theta(x) - \theta(\bar{x})| \leq L|x-\bar{x}|.
\end{align*}
In addition the functions $h$ and $\theta$ are assumed to be bounded.

\begin{thm} 
Suppose that all conditions listed in {\it (A1')}, {\it (A2)}, {\it (A3)}, {\it (A4)} or {\it (A1')}, {\it (A2)}, {\it (A3')}, {\it (A4)} are satisfied.  Then there exists $G \in \mathcal{C}^{2,1}(\mathbb{R} \times [0,T)) \cap \mathcal{C}(\mathbb{R} \times [0,T]) $, a classical positive solution  to equation \eqref{thirdequation}. In addition, the solution $G$ is uniformly bounded together with $D_{x} G$.
\end{thm} 
 
 \begin{proof} The proof is a straightforward repetition of the  proof of Theorem \ref{maintheorem}.
 \end{proof}

\subsection*{Acknowledgments} I would like to gratefully acknowledge the work carried out by the Referee.


\begin{thebibliography}{HD}







\bibitem{Benedetti} G. Benedetti, L. Campi,  \emph{Utility indifference valuation for non-smooth payoffs with an application to power derivatives}, Appl. Math. Optim. 73 (2016), 349--389.

\bibitem{Benth} F. E. Benth, K. H. Karlsen,
\emph{A PDE representation of the density of the minimal entropy
martingale measure in stochastic volatility markets}, Stochastics, 77 (2005), 109--137.

\bibitem{Berdjane} B. Berdjane,  S. Pergamenshchikov, \emph{Optimal consumption and investment for markets
with random coefficients}. Finance Stoch. 17 (2013), 419--446.

\bibitem{Besnousan} A. Bensoussan,  J. Frehse,  H. Nagai, \emph{Some results on risk-sensitive control with full observation},
Appl. Math. Optim.
37 (1998), 1--41.


\bibitem{Davis1} M. Davis, S. Lleo, \emph{Jump-diffusion risk-sensitive asset management I: diffusion factor model}, SIAM
J. Financ. Math. 2 (2011), 22--54.

\bibitem{Davis2}  M. Davis, S. Lleo, \emph{Jump-diffusion risk-sensitive asset management II: jump-diffusion factor model},
SIAM J. Control Optim. 51 (2013), 1441--1480.





\bibitem{Hernandez} N. Castaneda-Leyva, D. Hern\'{a}ndez-Hern\'{a}ndez, \emph{Optimal consumption investment problems in incomplete markets with stochastic coefficients},
SIAM J. Control Optim. 44 (2005), 1322--1344. 


\bibitem{Delarue}  F.  Delarue, G. Guatteri, \emph{Weak existence and uniqueness for fbsdes},
Stoch. Process. Appl., 116 (2006), 1712--1742.

 \bibitem{Duffie} D. Duffie, P.L. Lions, \emph{PDE solutions of stochastic differential utility}, J. Math. Econ. 21 (1992), 577--606

\bibitem{DuffieEpstein}D. Duffie, L.G. Epstein, \emph{Stochastic differential utility}, Econometrica 60 (1992), 353--394.

\bibitem{Epstein} L.G. Epstein, S.E. Zin, \emph{Substitution, risk aversion, and the temporal behavior of consumption and
asset returns: a theoretical framework}, Econometrica 57 (1989), 937--969 

\bibitem{Fleming} W. H. Fleming, D. Hern\'{a}ndez-Hern\'{a}ndez,
\emph{An optimal consumption model with stochastic volatility}, Finance Stoch. 7 (2003), 245--262.

\bibitem{Fleming2}  W. H. Fleming, D. Hern\'{a}ndez-Hern\'{a}ndez,
\emph{The Tradeoff between Consumption and Investment in Incomplete Financial Markets}, Appl. Math. Optim. 52 (2005), 219--235.

\bibitem{mcefle} W. H. Fleming,  W. M. McEneaney,   
\emph{Risk-sensitive control on an infinite time horizon}, SIAM J. Control. Optim. 33 (1995), 1881--1915.

\bibitem{Gobet}  E. Gobet, J. Lemor, X. Warin, \emph{A regression-based Monte Carlo method to solve backward stochastic differential equations}, Ann. Appl. Probab. 15 (2005), 2172--2202.

\bibitem{Grasselli}  M. R. Grasselli, T. R.  Hurd,  \emph{Indifference pricing and hedging for volatility derivatives}, Appl. Math. Finance 14 (2007), 303--317.

\bibitem{Krylov} I. Gy\"{o}ngy, N. Krylov, \emph{Existence of strong solutions for It\^{o} s stochastic equations via approximations}, Probab. Theory Related Fields 105 (1996), 143--158.


\bibitem{Hata2} H. Hata, \emph{Risk-sensitive asset management in a general diffusion factor model: risk-seeking case}, Japan J. Indust. Appl. Math. 34 (2017), 59--98. 

\bibitem{Hata3}  H. Hata,  H. Nagai, S. J. Sheu, \emph{An optimal consumption problem for general factor models}, SIAM J. Control Optim. 56 (2018), 3149--3183.

\bibitem{Hata}  H. Hata, S. J. Sheu, \emph{On the Hamilton--Jacobi--Bellman equation for an optimal consumption problem: I. Existence of solution}, SIAM J. Control Optim., 50 (2012), 2373--2400.




\bibitem{Henderson} V.  Henderson, \emph{Valuation of claims on nontraded assets using utility maximization}, Math. Finance 12 (2002), 351--373.



\bibitem{Hernandez3} D. Hern\'{a}ndez-Hern\'{a}ndez, \emph{ Variance-Optimal Martingale Measures for Diffusion
Processes with Stochastic Coefficients},  D. Set-Valued Var. Anal 26 (2018), 975--991. 

\bibitem{Hernandez2} D. Hern\'{a}ndez-Hern\'{a}ndez,  S. J. Sheu, \emph{Solution of the HJB Equations Involved in Utility-Based Pricing}, In XI Symposium on Probability and Stochastic Processes, Springer International Publishing, (2015) 177--198.

\bibitem{Jacka} S. Jacka,  A. Mijatovi\'{c}, \emph{On the policy improvement algorithm in continuous time}, Stochastics 89 (2017), 348--359.



\bibitem{Kraft2} H. Kraft, F.T. Seifried,  M. Steffensen, \emph{Consumption-portfolio optimization with recursive utility in
incomplete markets}, Finance Stoch. 17 (2013), 161--196.

\bibitem{Kraft} H. Kraft, T. Seiferling, F. T. Seifried, \emph{Optimal consumption and investment with Epstein-Zin recursive utility}, Finance Stoch., 21 (2017), 187--226.

\bibitem{Laurent} J.P. Laurent,  H. Pham, \emph{Dynamic programming and mean-variance hedging}, Finance  Stoch. 3 (1999), 82--110.

\bibitem{Matoussi} A.  Matoussi, H.  Xing, \emph{Convex  duality  for  Epstein-Zin  stochastic  differential  utility}, Math. Finance 28 (2018), 991--1019.

\bibitem{Musiela} M. Musiela, T. Zariphopoulou , \emph{An Example of Indifference Prices under Exponential Preferences}, Finance  Stoch. 8 (2004), 229--239.


\bibitem{Nagai2} H. Nagai, \emph{Optimal strategies for risk--sensitive portfolio optimization problems for general factor
models}, SIAM J. Control Optim.
41 (2003), 1779--1800.

\bibitem{Nagai} H. Nagai, \emph{H-J-B Equations of Optimal Consumption-Investment and Verification Theorems}, Appl. Math. Optim., 71 (2015), 279--311.


\bibitem{Palczewski} A. Palczewski, \emph{Risk minimizing strategies for a portfolio of interest-rate securities}, 
Banach Center Publ. 83 (2008), 195--212.

\bibitem{Pham2} H. Pham, \emph{Continuous-time Stochastic Control and Optimization with Financial Applications}, Stochastic Modelling and Applied Probability, Springer-Verlag, Berlin, (2009).

\bibitem{Pham} H. Pham,  \emph{Smooth solutions to optimal investment models with stochastic volatilities and portfolio constraints}, Appl. Math. Optim. 46 (2002),  55--78.

\bibitem{Sircar} R. Sircar,  T. Zariphopoulou, \emph{Bounds and asymptotic approximations for utility prices when volatility is random}, SIAM J. Control Optim. 43 (2004), 1328--1353.

\bibitem{Stroock} D. W. Stroock , S. R. S. Varadhan, \emph{Multidimensional diffusion processes}, Springer-Verlag Berlin Heidelberg, (2006).

\bibitem{Trybula2} J. Trybu{\l}a, \emph{Optimal consumption problem in the Vasicek model},  Opuscula  Math. 35 (2015), 547--560.

\bibitem{Trybula} J. Trybu{\l}a, D. Zawisza, \emph{Continuous-time portfolio choice under monotone mean-variance preferences{\textemdash}stochastic factor case},  Math. Oper. Res. 44 (2019), 966--987.



\bibitem{Xing}  H. Xing, \emph{Consumption--investment optimization with Epstein-Zin utility in incomplete markets}, Finance
Stoch. 21 (2017), 227--262.


\bibitem{Zariphopoulou}   T. Zariphopoulou, \emph{A solution approach to valuation with unhegeable risks}, Finance  Stoch. 5 (2001), 61--82.

\bibitem{Zawisza2} D. Zawisza, \emph{Existence results for Isaacs equations with local conditions and related semilinear Cauchy problems},  Ann. Polon. Math. 121 (2018), 175--196.

\bibitem{Zawisza} D. Zawisza,
\emph{Robust Consumption-Investment Problem on Infinite Horizon}, Appl. Math. Optim.,  72 (2015), 469--491. 

\bibitem{Zawisza4}  D. Zawisza,  \emph{Robust portfolio selection under exponential preferences}, Appl. Math. (Warsaw) 37 (2010), 215--230.

\bibitem{Zawisza3} D. Zawisza, \emph{Target achieving portfolio under model misspecification: quadratic optimization framework}, Appl. Math. (Warsaw) 39 (2012), 425--443.

\end{thebibliography}
\end{document}